\documentclass[12pt]{article}
\usepackage{graphicx}
\usepackage{amsmath,amsthm,amssymb,enumerate}
\usepackage{euscript,mathrsfs}
\usepackage{amsmath}
\usepackage{color}
\usepackage{dsfont}
\usepackage{mathrsfs}
\usepackage[left=2cm,right=2cm,top=3.5cm,bottom=3.5cm]{geometry}
\usepackage[framemethod=tikz]{mdframed}
\usepackage{bm}

\usepackage{comment}

\usepackage{soul}

\catcode`\@=11 \@addtoreset{equation}{section}

\catcode`\@=12

\newtheorem{Theorem}{Theorem}[section]
\newtheorem{Proposition}[Theorem]{Proposition}
\newtheorem{Lemma}[Theorem]{Lemma}
\newtheorem{Corollary}[Theorem]{Corollary}

\theoremstyle{definition}
\newtheorem{Definition}[Theorem]{Definition}

\newtheorem{Remark}[Theorem]{Remark}

\newcommand{\bTheorem}[1]{
\begin{Theorem} \label{T#1} }
\newcommand{\eT}{\end{Theorem}}

\newcommand{\bProposition}[1]{
\begin{Proposition} \label{P#1}}
\newcommand{\eP}{\end{Proposition}}

\newcommand{\bLemma}[1]{
\begin{Lemma} \label{L#1} }
\newcommand{\eL}{\end{Lemma}}

\newcommand{\bCorollary}[1]{
\begin{Corollary} \label{C#1} }
\newcommand{\eC}{\end{Corollary}}

\newcommand{\bRemark}[1]{
\begin{Remark} \label{R#1} }
\newcommand{\eR}{\end{Remark}}

\newcommand{\bDefinition}[1]{
\begin{Definition} \label{D#1} }
\newcommand{\eD}{\end{Definition}}

\newcommand{\Del}{\Delta_x}

\newcommand{\bfphi}{\boldsymbol{\varphi}}

\newcommand{\bFormula}[1]{
\begin{equation} \label{#1}}
\newcommand{\eF}{\end{equation}}

\newcommand{\Ov}[1]{\overline{#1}}

\newcommand{\aleq}{\stackrel{<}{\sim}}

\newcommand{\ageq}{\stackrel{>}{\sim}}

\newcommand{\vr}{\varrho}

\newcommand{\tvt}{\tilde \vt}

\newcommand{\vt}{\vartheta}
\newcommand{\vu}{\vc{u}}

\newcommand{\vc}[1]{{\bm #1}}

\newcommand{\Div}{{\rm div}_x}
\newcommand{\Grad}{\nabla_x}

\newcommand{\dx}{\,{\rm d} {x}}

\newcommand{\dt}{\,{\rm d} t }

\newcommand{\intR}[1]{\int_{R^3} #1 \ \dx}

\newcommand{\D}{{\rm d}}

\newcommand{\br}{ \nonumber \\ }

\def\softd{{\leavevmode\setbox1=\hbox{d}%
	\hbox to 1.05\wd1{d\kern-0.4ex{\char039}\hss}}}
\definecolor{Cgrey}{rgb}{0.85,0.85,0.85}
\definecolor{Cblue}{rgb}{0.50,0.85,0.85}
\definecolor{Cred}{rgb}{1,0,0}
\definecolor{fancy}{rgb}{0.10,0.85,0.10}

\newcommand\Cbox[2]{%
\newbox\contentbox%
\newbox\bkgdbox%
\setbox\contentbox\hbox to \hsize{%
	\vtop{
		\kern\columnsep
		\hbox to \hsize{%
			\kern\columnsep%
			\advance\hsize by -2\columnsep%
			\setlength{\textwidth}{\hsize}%
			\vbox{
				\parskip=\baselineskip
				\parindent=0bp
				#2
			}%
			\kern\columnsep%
		}%
		\kern\columnsep%
	}%
}%
\setbox\bkgdbox\vbox{
	\color{#1}
	\hrule width  \wd\contentbox %
	height \ht\contentbox %
	depth  \dp\contentbox
	\color{black}
}%
\wd\bkgdbox=0bp%
\vbox{\hbox to \hsize{\box\bkgdbox\box\contentbox}}%
\vskip\baselineskip%
}

\mdfdefinestyle{MyFrame}{%
linecolor=black,
outerlinewidth=1pt,
roundcorner=5pt,
innertopmargin=\baselineskip,
innerbottommargin=\baselineskip,
innerrightmargin=10pt,
innerleftmargin=10pt,
backgroundcolor=white!20!white}

\begin{document}

\title{On the long time behaviour of solutions to the Navier--Stokes--Fourier system on unbounded domains}

\author{Elisabetta Chiodaroli$^\dag$ \and Eduard Feireisl$^\ddag$ \thanks{The authors acknowledge support by MIUR, within PRIN2022T9K54B ``Classical equations of compressible fluid mechanics: existence and properties of non-classical solutions”. 
The work of EF was also partially supported by the Czech Sciences Foundation (GA\v CR), Grant Agreement 24--11034S. The Institute of Mathematics of the Academy of Sciences of the Czech Republic is supported by RVO:67985840. }  } 

\date{}

\maketitle

\bigskip
\centerline{$^\dag$ Dipartimento di Matematica, Universit\` a di Pisa} 
	
\centerline{Via F. Buonarroti 1/c, 56127 Pisa, Italy}
\medskip

\centerline{$^\ddag$   Institute of Mathematics of the Academy of Sciences of the Czech Republic}

\centerline{\v Zitn\' a 25, CZ-115 67 Praha 1, Czech Republic}

\begin{abstract}
We consider the Navier--Stokes--Fourier system on an unbounded domain in the Euclidean space $R^3$, supplemented by the far field conditions for the phase variables, specifically: $\vr \to 0,\ \vt \to \vt_\infty, \ \vu \to 0$ as $\ |x| \to \infty$. We study the long time behaviour of solutions and we prove that any global-in-time weak solution to the NSF system approaches the equilibrium $\vr_s = 0,\ \vt_s = \vt_\infty,\ \vu_s = 0$ in the sense of ergodic averages for time tending to infinity. As a consequence of the convergence result combined with the total mass conservation, we can show that the total momentum of global-in-time weak solutions is never globally conserved.
\end{abstract}

\bigskip

\section{Introduction}
\label{i}

The Navier--Stokes--Fourier (NSF) system is frequently used in astrophysics as a model of the time evolution of gaseous stars or stellar interiors, see e.g. 
Battaner \cite{BATT}, the collection \cite{ThoChDa}, or Gough \cite{Gough} to name only a few relevant references. In this context, the natural physical space is an unbounded domain in the Euclidean space $R^3$ and the system of field equations must be supplemented by the far field conditions for the phase variables. Our goal is to study the long time behaviour of solutions in a suitable phase space. We focus on certain mathematical aspects of this problem whereas the model system of equations is rather simplified with respect 
to those used in the real world applications.  

\subsection {Navier--Stokes--Fourier system}
\label{b}

For the mass density $\vr = \vr(t,x)$, the velocity $\vu = \vu(t,x)$, and the temperature $\vt = \vt(t,x)$, we consider the following NSF system of field equations:
\begin{align} 
	\partial_t \vr + \Div (\vr \vu) &= 0, \label{b1} \\ 
	\partial_t (\vr \vu) + \Div (\vr \vu \otimes \vu) + 
	\Grad p(\vr, \vt) &= \Div \mathbb{S}(\vt, \Grad \vu), \label{b2}\\
	\partial_t (\vr e(\vr, \vt)) + \Div (\vr e(\vr, \vt) \vu) + \Div \vc{q}(\vt, \Grad \vt) &= 
	\mathbb{S}(\vt, \Grad \vu) : \Grad \vu  - p(\vr, \vt) \Div \vu  + \Lambda_r(\vt, \vt_\infty),     \label{b3}
\end{align}	
$t > 0$, $x \in R^3$, supplemented with the far field conditions
\begin{align}
	\vr &\to 0,\ \vt \to \vt_\infty, \ \vu \to 0 \ \mbox{as} \ |x| \to \infty, \label{b4} 
	\end{align}
where $\vt_\infty > 0$ is a positive constant. Our aim is to study the large time dynamics for this system with zero constant density and velocity prescribed at infinity.
The viscous stress tensor $\mathbb{S}$ is given by Newton's rheological law
\begin{equation} \label{b4a}
	\mathbb{S}(\vt, \Grad \vu) = \mu (\vt) \left( \Grad \vu + \Grad \vu^t - 
	\frac{2}{3} \Div \vu \mathbb{I} \right) + \eta(\vt) \Div \vu \mathbb{I}.
	\end{equation}	 
The heat flux $\vc{q}$ obeys Fourier's law 
\begin{equation} \label{b4b}
	\vc{q}(\vt, \Grad \vt) = - \kappa (\vt) \Grad \vt.
\end{equation}	

Motivated by models in astrophysics, we implement the effects of radiation in the constitutive relations, cf. e.g. Oxenius \cite{OX}. Specifically, the pressure $p$ as well as the internal energy $e$ are augmented by radiation components 
\begin{align} 
	p(\vr, \vt) &= p_m (\vr, \vt) + p_r(\vt),\ p_r(\vt) = \frac{a}{3} \vt^4, 
\label{b4c} \\ 
e(\vr, \vt) &= e_m (\vr, \vt) + e_r (\vr, \vt), \ e_r(\vr, \vt) =  a \frac{\vt^4}{\vr}.
\label{b4d}
\end{align}	
We recall that the subscripts $m$ and $r$ stand for \textit{molecular} and \textit{radiative} respectively.	
Similarly, we consider the heat conductivity coefficient in the form 
\begin{equation} \label{b4e}
	\kappa (\vt) = \kappa_m (\vt) + \kappa_r \vt^3.
\end{equation}
Finally, the heat source $\Lambda_r$ in the internal energy balance \eqref{b3} 
represents a radiative cooling/heating with respect to the equilibrium temperature $\vt_\infty$, 
\begin{equation} \label{b4f}
\Lambda_r (\vt, \vt_\infty) = - \lambda (\vt^4 - \vt^4_\infty).	
	\end{equation}
	
The regularizing effect of radiation was recognized in \cite{DF1} and represents an indispensable ingredient of the existence theory for the NSF system developed in \cite{FeNo6A}, \cite{FeiNovOpen}. In fact, the results obtained in the present paper are independent of the presence of radiation and remain valid even in the case $a = \lambda = 0$.
However, the existence of global--in--time weak solutions in absence of radiative effects is an outstanding open problem, hence it is worth focusing the attention on the case including radiation for which the theory of existence is settled. 

Remarkably, the convergence to equilibria shown herein remains valid even in the barotropic regime and it represents a novelty even in that case, since we both prescribe zero far field density and we can deal with a linear behavior of the pressure as function of $\rho$ near $\rho=0$, while so far available results for the compressible barotropic Navier--Stokes equations concern only pressure laws of polytropic type ($p(\rho)=\rho^\gamma$, $\gamma >3/2$), see \cite{FP9} or see also \cite{NovPok07}, where non-zero constant density at infinity is considered.

\subsection{Convergence to equilibrium}

Besides the far field conditions \eqref{b4}, our main working hypothesis is 
that the total mass of the fluid is finite, 
\begin{equation} \label{b6a}
	\intR{ \vr(t, \cdot) } = m_0 > 0. 
\end{equation}
It is easy to check the problem \eqref{b1}--\eqref{b4} admits a static solution 
\begin{equation} \label{b6f}
	\vr_s = 0,\ \vt_s = \vt_\infty,\ \vu_s = 0.
\end{equation} 	 
Our main goal is to show that any (global--in--time) weak solution to the NSF system \eqref{b1}--\eqref{b4} approaches the equilibrium, 
\begin{equation} \label{b6g}
\vr(t, \cdot) \to 0,\ \vt(t, \cdot) \to \vt_\infty,\ 
\vu(t, \cdot) \to 0 \ \mbox{as}\ t \to \infty 
\end{equation}
in a certain sense. Despite the finite mass constraint \eqref{b6a}, the density 
may converge to zero in the long run due to possible loss of mass ``at infinity''. Indeed a similar result in the context of \emph{isentropic} Navier--Stokes system was shown in \cite{FP9}. 

The strategy employed in \cite{FP9} is based on the energy inequality with a \emph{non--negative} energy. In particular, the energy represents a Lyapunov function. The substantial difficulty of the temperature dependent case is the absence of any Lyapunov function. Indeed, as shown below, the energy inequality in the present setting must be replaced by the ballistic energy balance. The ballistic energy, however, may, and actually does (see Section \ref{C}), become negative and even unbounded from below as $t \to \infty$. The convergence claimed in \eqref{b6g} therefore holds only in the sense of ergodic averages. More specifically, we will show that 
\begin{align} 
\frac{1}{T}	\int_0^T \left( \| \vr(t, \cdot) \|_{L^{\frac{5}{3}}(R^3)}^{\frac{5}{3}} +
\| \vt(t, \cdot) - \vt_\infty \|^2_{L^6(R^3)} +  
\| \vu(t, \cdot) \|^2_{L^6(R^3; R^3)}   \right) \dt \to 0 \ \mbox{as}\ T \to \infty
\label{b6h}
	\end{align}
for \emph{any} weak solution of the NSF system \eqref{b1}--\eqref{b4}.

There is a corollary of the convergence result \eqref{b6h} that may be of independent interest. It is easy to see that \eqref{b6h} yields a sequence of times $T_n \to \infty$ such that 
\begin{equation} \label{b6i}
\int_{T_n}^{T_n + 1} \left( \| \vr(t, \cdot) \|_{L^{\frac{5}{3}}(R^3)}^{\frac{5}{3}} +
\| \vt(t, \cdot) - \vt_\infty \|^2_{L^6(R^3)} +  
\| \vu(t, \cdot) \|^2_{L^6(R^3; R^3)}   \right) \dt \to 0 \ \mbox{as}\ n \to \infty.
\end{equation}
This convergence, combined with the total mass conservation \eqref{b6a}, implies 
\begin{equation} \label{b6j}
\int_{T_n}^{T_n + 1} \| \vr \vu (t, \cdot) \|_{L^1(R^3)} \dt \to 0 
\ \mbox{as}\ n \to \infty.	  
\end{equation}
In particular, the total momentum 
\[
\intR{ \vr \vu (t, \cdot) } = \vc{P} 
\]
is never globally conserved unless $\vc{P} = 0$. Thus the blow--up criteria for smooth solutions claimed by Rozanova \cite{Roza} or 
Xin \cite{XIN} based on fast decay of the velocity field are violated even in the class of weak solutions that, however, exist globally in time.

The main goal of the paper is to establish the convergence result \eqref{b6h}. 
In Section \ref{W}, we introduce the weak formulation of the problem, formulate the main constitutive restrictions, and recall the existing results on global--in--time weak solutions. Section \ref{M} contains the main result of the paper. In Section \ref{D}, we derive the uniform bounds necessary for the proof of convergence. Section \ref{L} completes the proof of the main result. 
The paper is concluded by a brief discussion in Section \ref{C}.

\section{Principal hypotheses, weak solutions}
\label{W}

As we are concerned with the long--time behaviour of weak solutions, the specific form of initial data does not play any role in the analysis. Consequently, 
we introduce the concept of weak solution defined on an open time interval $(0, \infty)$. 

\subsection{Weak solutions}

The concept of weak solution is based on the theory developed in \cite{ChauFei}, \cite{FeiNovOpen} for \emph{bounded} spatial domains. The approach developed therein is based on replacing the internal
energy equation by the entropy inequality which is consistent with the dissipative character of the fluid flow.

\subsubsection{Mass conservation}
\label{W1}

Mass conservation is enforced through a decaying velocity field belonging to the class
\begin{equation} \label{b6}
	\vu \in L^2_{\rm loc}(0, \infty; D^{1,2}_0(R^3; R^3)).
\end{equation}
Here the symbol $D^{1,2}_0(R^3; R^3)$ denotes the homogeneous Sobolev space -- 
the closure of $C^\infty_c(R^3; R^3)$ under the gradient norm $\| \Grad \vc{v} \|_{L^2(R^3; R^{3 \times 3})}$.

We suppose the density is a measurable function
satisfying
\begin{equation} \label{b7}
\vr \geq 0 , \ \intR{\vr(t, \cdot) } = m_0 > 0 \ \mbox{for all}\ t > 0.
\end{equation}

Finally, we replace the equation of continuity \eqref{b1} by its renormalized version satisfied in the weak sense, 
\begin{equation} \label{b8}
	\int_0^\infty \intR{ \Big[ B(\vr) \left( \partial_t \varphi + \vu \cdot \Grad \varphi \right) - b(\vr) \Div \vu \varphi \Big] } \dt = 0
	\end{equation}
for any $\varphi \in C^1_c((0,\infty) \times R^3)$, and any 
$b \in BC[0, \infty)$, $B \in C[0, \infty) \cap C^1(0, \infty)$ satisfying 
\[
B'(\vr) \vr - B(\vr) = b(\vr). 
\]
Note that $B(\vr) = \vr$, $b \equiv 0$ is a legitimate choice; whence \eqref{b8} includes the weak formulation of  \eqref{b1}.

\subsubsection{Momentum balance}

The momentum equation \eqref{b2} is replaced by the integral identity
\begin{equation} \label{b9}
	\int_0^\infty \intR{ \left[ \vr \vu \cdot \Big(\partial_t \bfphi + 
		\vu \cdot \Grad \bfphi \Big)  + p(\vr, \vt) \Div \bfphi \right] } \dt 
	= \int_0^\infty \intR{ \mathbb{S}(\vt, \Grad \vu) : \Grad \bfphi } \dt	 
\end{equation}		
for any $\bfphi \in C^1_c((0,\infty) \times R^3; R^3)$.

\subsubsection{Entropy balance inequality}

The internal energy equation \eqref{b3} is not convenient for a weak formulation because of the lack of suitable {\it a priori} bounds. Instead, we use the entropy balance 
\[
\partial_t \vr s(\vr, \vt) + \Div (\vr s(\vr, \vt)\vu) + \Div \left( \frac{ \vc{q} (\vt, \Grad\vt)}{\vt} \right) = \sigma, 
\]
\[
\sigma \geq \frac{1}{\vt} \Big[ \mathbb{S}(\vt, \Grad \vu) : \Grad \vu - \frac{\vc{q}(\vt,\Grad \vt) \cdot \Grad \vt }{\vt} \Big] + \frac{1}{\vt} \Lambda_r (\vt, \vt_\infty), 
\]
where the entropy $s = s(\vr, \vt)$ is determined in terms of $p$ and $e$ by means of Gibbs' equation
\begin{equation} \label{b9a}
\vt D s = De + p D \left( \frac{1}{\vr} \right).	
	\end{equation} 
The corresponding weak formulation reads
\begin{equation} \label{b10}
\vt > 0, \quad (\vt - \vt_\infty) \in L^2_{\rm loc}(0, \infty ; W^{1,2}(R^3))	
\end{equation}
\begin{align}
\int_0^\infty &\intR{ \Big[ \vr s(\vr, \vt) \left( \partial_t \varphi + \vu \cdot \Grad \varphi \right) + \frac{ \vc{q}(\vt,\Grad \vt) }{\vt} \cdot \Grad \varphi \Big] } \dt \br
&\leq - \int_0^\infty \intR{ \frac{\varphi}{\vt} \Big[ \mathbb{S}(\vt, \Grad \vu) : \Grad \vu - \frac{\vc{q}(\vt, \Grad \vt) \cdot \Grad \vt }{\vt} \Big] } \dt  	- 
\int_0^\infty \intR{ \frac{1}{\vt} \Lambda_r (\vt, \vt_\infty) \varphi } \dt
\label{b11}
		\end{align}
for any $\varphi \in C^1_c ((0, \infty) \times R^3)$, $\varphi \geq 0$.

\subsubsection{Ballistic energy balance}
\label{W4}

Adopting the concept of weak solution introduced in \cite[Chapter 12]{FeiNovOpen}
we close the weak formulation by imposing the ballistic energy balance
\begin{align}
	- &\int_0^\infty \partial_t \psi \intR{ \left[ \frac{1}{2} \vr |\vu|^2 + \vr e(\vr, \vt) - 
		\tvt \vr s(\vr, \vt) + \frac{a}{3} \vt_\infty^4 \right] } \dt \br
	&+ \int_0^\infty \psi \intR{ \frac{\tvt}{\vt} \left(\mathbb{S}(\vt, \Grad \vu) : \Grad \vu - \frac{\vc{q} \cdot \Grad \vt}{\vt} 
		\right)} + \lambda \int_0^\infty \psi \intR{ \left(\frac{\vt - \tvt}{\vt} \right) (\vt^4 - \vt_\infty^4)         }\dt \br
	&\leq - \int_0^\infty \psi \intR{ \left[ \vr s(\vr, \vt) \left( \partial_t \tvt + \vu \cdot \Grad \tvt \right) + \frac{\vc{q}(\vt, \Grad \vt)}{\vt} \cdot \Grad \tvt		\right] } \dt
	\label{b12}
\end{align}
for any $\psi \in C^1_c(0, \infty)$, and any
$\tvt \in C^1((0,\infty) \times R^3)$, 
$(\tvt - \vt_\infty) \in C^1_c ([0, \infty) \times R^3)$.

\subsubsection{Weak solutions on unbounded domains}

The theory of weak solutions specified in Sections \ref{W1}--\ref{W4} was developed in \cite{ChauFei} and \cite[Chapter 12]{FeiNovOpen} in the case of bounded physical domain $\Omega$. The extension to the whole space $R^3$ can be 
done by the method of invading domains 
\[
\Omega_R = \{ |x| < R \} 
\]
with the boundary conditions
\begin{equation} \label{BC}
\vt|_{\partial \Omega_R} = \vt_\infty,\ \vu|_{\partial \Omega_R} = 0 
\end{equation}
sending $R \to \infty$. The process is successful as long as suitable uniform bounds independent of $R$ are available and has been performed e.g. by 
Jessl\' e, Jin, Novotn\' y \cite{JeJiNo} or Poul \cite{Poul1}. As observed by Poul \cite{Poul1}, the extension of the {\it a priori} bounds resulting from the ballistic energy balance requires additional integrability properties to be imposed on the density, namely 
\begin{equation} \label{b13a}
	\vr \log(\vr) \in L^\infty_{\rm loc}(0, \infty; R^3). 
	\end{equation}

It is mainly the slow growth of the function $\vr \log (\vr)$ for $\vr \to 0$ that is relevant on unbounded domains. 
We introduce a cut--off function 
\[
T_k(\vr) = \min \{ \vr, k \},
\]
together with
\begin{equation} \label{b13b}
L_k (\vr) = \int_1^\vr \frac{T_k(z)}{z^2} \ \D z.
\end{equation}
It is easy to check that
\[
(\vr L_k(\vr))' \vr -  \vr L_k(\vr) = T_k(\vr)
\]
and, by virtue of the renormalized equation \eqref{b8},  
\begin{equation} \label{b13}
	\partial_t (\vr L_k(\vr)) + \Div (\vr L_k(\vr) \vu) + T_k(\vr) \Div \vu = 0 
\end{equation}	 
in $\mathcal{D}'((0, \infty) \times R^3)$. 
As $(\vr, \vu)$ enjoy the integrability properties specified in \eqref{b6}, \eqref{b13a}, we may ``integrate''  \eqref{b13} to justify the relation 
\begin{equation} \label{b14}
	\frac{\D }{\dt} \intR{ \vr L_k(\vr) } = - \intR{ T_k(\vr) \Div \vu }
	\ \mbox{in}\ \mathcal{D}'(0, \infty).
\end{equation}
The integral on the right--hand side can be controlled as
\begin{equation} \label{b15}
	\left| \intR{ T_k(\vr) \Div \vu } \right| \leq 
	\| T_k(\vr) \|_{L^2(R^3)} \| \Div \vu \|_{L^2(R^3)}, 
\end{equation}
where
\begin{equation} \label{b16}
	\| T_k(\vr) \|_{L^2(R^3)} \leq \sqrt{k} \intR{\vr} = \sqrt{k} \ m_0.
\end{equation}

\subsection{Constitutive relations}
\label{cc}

The hypotheses imposed on constitutive relations are those required by the existence theory developed in \cite{FeiNovOpen}. As already mentioned in the introduction, the pressure, the internal energy, and the entropy are augmented by radiations components. Besides Gibbs' relation \eqref{b9a}, we impose the hypothesis of thermodynamic stability 
\begin{equation} \label{c1a}
	\frac{\partial p(\vr, \vt) }{\partial \vr} > 0,\ 
	\frac{\partial e(\vr, \vt) }{\partial \vt} > 0.
	\end{equation}
	
\subsubsection{Equations of state}

The molecular component $p_m$ of the pressure satisfies 
a general equation of state (EOS) of mono--atomic gases, 	
\begin{equation} \label{cc1}
	p(\vr, \vt) = p_m(\vr, \vt) + p_r(\vt),\ 
	p_m (\vr, \vt) = \frac{2}{3} \vr e_m(\vr, \vt),\ p_r(\vt) = \frac{a}{3} \vt^4,\ 
	a > 0,
\end{equation}
Consequently, Gibbs' relation \eqref{b9a} yields 
\begin{align} 
	p_m (\vr, \vt) &= \vr \vt \frac{P(Z)}{Z},\ \mbox{where we have set}\ Z = \frac{\vr}{\vt^{\frac{3}{2}}},	\label{cc2} \\
	e (\vr, \vt) &= e_m(\vr, \vt) + e_r(\vr, \vt),\ 
	e_m (\vr, \vt) = \frac{3}{2} \vt \frac{P(Z)}{Z}, \ e_r = a \frac{\vt^4}{\vr}, 
	\label{cc3} \\
	s(\vr, \vt) &= s_m (\vr, \vt) + s_r (\vr, \vt),\ 
	s_m (\vr, \vt) = \mathcal{S} (Z),\ s_r (\vr, \vt) = \frac{4a}{3} \frac{\vt^3}{\vr}, \br
\mbox{where}\	\mathcal{S}'(Z) &= - \frac{3}{2} \frac{ \frac{5}{3} P(Z) - P'(Z) Z   }{Z^2}.  
\label{cc4}	
\end{align}
The thermodynamic stability hypothesis \eqref{c1a} amounts to require that
\begin{equation} \label{cc5}
	P \in C^1[0, \infty),\ P'(Z) > 0,\  
	\frac{5}{3} P(Z) - P'(Z) Z > 0 \ \mbox{for any}\ Z > 0.
\end{equation}

In addition, we suppose the fluid behaves like perfect gas in the non--degenerate area $Z << 1$, specifically, 
\begin{equation} \label{cc6}
	P \in C^2[0, \infty),\ P''(0) = 0,\ \frac{P(Z)}{Z} \geq \underline{P} > 0 \ \mbox{for all}\ Z > 0.  
\end{equation}				 
Finally, we impose the Third law of thermodynamics relevant for $Z>>1$, 
\begin{equation} \label{cc7} 
	\mathcal{S}(Z) \to 0 \ \mbox{as}\ Z \to \infty, \\ 
	\frac{\frac{5}{3} P(Z) - P'(Z) Z}{Z} < \Ov{P} \ \mbox{for all}\ Z > 0.
\end{equation}

\subsubsection{Transport coefficients}

The viscous stress is given by Newton's rheological law \eqref{b4a}, 
and the heat flux $\vc{q}$ obeys the Fourier law \eqref{b4b}, where the diffusion
transport coefficients satisfy 
\begin{align}
	0 < \underline{\mu} (1 + \vt) &\leq \mu(\vt) \leq \Ov{\mu}(1 + \vt) ,\ \mu' \in BC[0, \infty), \label{cc10} \\ 	
	0  &\leq \eta(\vt) \leq \Ov{\eta}(1 + \vt), \label{cc11}\\
	\kappa(\vt) = \kappa_m (\vt) + \kappa_r \vt^3,\ \kappa_r > 0,\ 
	0 < \underline{\kappa}(1 + \vt) &\leq \kappa_m (\vt) \leq \Ov{\kappa}(1 + \vt).
	\label{cc12}
\end{align}

\subsection{Global--in--time weak solutions}

Under the hypotheses specified in Section \ref{cc}, the existence of global--in--time weak solutions of the NSF system on a bounded domain with the Dirichlet boundary conditions \eqref{BC} was established in \cite[Chapter 12, Theorem 18]{FeiNovOpen}. The results can be extended to unbounded domains by the method of invading domains and the uniform bounds specified in Section \ref{D}.

\section{Main result}
\label{M}

Having introduced the necessary preliminary material, we are ready to state the main result of the present paper.

\begin{mdframed}[style=MyFrame]
	
\begin{Theorem}[{\bf Long--time stability}] \label{MT1}	
	
Let $\vt_\infty > 0$ be a positive constant. Suppose that the EOS relating the pressure, density, and entropy as well as the transport coefficients satisfy the structural restrictions introduced 
Section \ref{cc}. Let $(\vr, \vt, \vu)$ be a weak solution of the 
NSF system in $(0, \infty) \times R^3$ in the sense specified in Sections \ref{W1}--\ref{W4}.

Then 
\[
\frac{1}{T} \int_1^T \intR{ \left[ \| \vr(t, \cdot) \|^{\frac{5}{3}}_{L^{\frac{5}{3}}(R)} + \left\| \vt(t, \cdot) - \vt_\infty \right\|_{W^{1,2}(R^3)}^2 + \| \vu (t, \cdot) \|^2_{D^{1,2}_0 (R^3; R^3)}  \right] } \dt \to 0
\]
as $T \to \infty$.
\end{Theorem}
	\end{mdframed}
	
The rest of the paper is devoted to the proof of Theorem \ref{MT1}. It will be clear that a slightly weaker result, namely 
\[
\frac{1}{T} \int_1^T \intR{ \left[ \| \vr(t, \cdot) \|^{\frac{5}{3}}_{L^{\frac{5}{3}}(R)} + \left\| \vt(t, \cdot) - \vt_\infty \right\|_{D^{1,2}_0(R^3)}^2 + \| \vu (t, \cdot) \|^2_{D^{1,2}_0 (R^3; R^3)}  \right] } \dt \to 0 
\]	 
as $T \to \infty$, can be shown even in the absence of the radiative 
terms, specifically for $a = \lambda = 0$. However, the \emph{existence} of global--in--time weak solutions in this case remains an outstanding open problem.

\section{Uniform bounds} 
\label{D}

We establish several uniform bounds for global--in--time weak solutions. 

\subsection{More about structural properties of the constitutive relations}
\label{S}

The following result was proved in \cite[Section 2.1.1]{FeiLuSun}. 

\begin{Lemma} \label{LS1}
	Under the hypotheses of Section \ref{cc}, the function $P$ determining the pressure $p_m$ can be decomposed as 
	\begin{equation} \label{S1}
		P(Z) = \Ov{p} Z^{\frac{5}{3}} + \widetilde{P}(Z),\ \mbox{where}\ \Ov{p}  > 0,	
	\end{equation}
	and 
	\begin{align} \label{S2}
		\widetilde{P} \in C^1[0, \infty),\ \widetilde{P}(0) = P(0) &= 0,\ \widetilde{P}'(0) = P'(0) > 0, \ \widetilde{P} \geq 0, \ \frac{\widetilde{P}(Z)}{Z} \to 0 \ \mbox{as}\ Z \to \infty, \br |\widetilde{P}'(Z)| &\leq c 
		\ \mbox{for all}\ Z \geq 0. 
	\end{align} 
\end{Lemma}

In view of \eqref{cc6}, \eqref{S1}, we have
\begin{equation} \label{S2a}
	\vr^{\frac{5}{3}} + \vr \vt \aleq p_m(\vr, \vt) = \frac{2}{3} \vr e_m(\vr, \vt).
\end{equation}
Here and hereafter, the symbol $a \aleq b$ means $a \leq c b$, where $c$ is a positive constant.

Next, we decompose the entropy into a bounded part and a logarithmic part.

\begin{Lemma} \label{LS2}
	Under the hypotheses of Section \ref{cc}, the function $\mathcal{S}$ determining the entropy $s_m$ can be decomposed as 
\begin{equation} \label{S3a}
\mathcal{S}(Z) = \mathcal{S}_B (Z) - P'(0) \log(Z) \mathds{1}_{Z \leq 1},
	\end{equation}	 
where 	
\begin{equation} \label{S3b}
\mathcal{S}_B \in BC[0, \infty),\ \mathcal{S}'_B (0) = 0,\ \lim_{Z \to \infty} 
\mathcal{S}_B (Z) = 0.
\end{equation}
	\end{Lemma}
	
\begin{proof}	 
Rewrite \eqref{cc4} in the form 
\[
\mathcal{S}'(Z) = - \frac{3}{2} \frac{ \frac{5}{3} \Big(P(Z) - P'(0) Z \Big) - \Big(P'(Z) - P'(0) \Big) Z }{Z^2} - \frac{P'(0)}{Z}.   
\]
Now, we may set 
\begin{equation} \label{S3}
	\mathcal{S}_B(Z) = \mathcal{S}(Z) + P'(0) \log(Z) \ \mbox{if}\ Z \leq 1, \ 
	\mathcal{S}_B(Z) = \mathcal{S}(Z) \ \mbox{if}\ Z > 1, 
\end{equation}
where 
\begin{equation} \label{S4}
	\mathcal{S}'_B(Z) = - \frac{3}{2}  \frac{ \frac{5}{3} \Big(P(Z) - P'(0) Z \Big) - \Big(P'(Z) - P'(0) \Big) Z }{Z^2} \to  \frac{1}{4} P''(0) = 0 \ \mbox{as}
	\ Z \to 0.
\end{equation}

\end{proof}

\subsection{Energy estimates}
\label{e}

Applying Korn and Sobolev inequalities, and using
the growth restrictions imposed through hypotheses \eqref{cc10}, \eqref{cc12}, we have 
\begin{equation} \label{D1}
\intR{	\frac{1}{\vt} \mathbb{S} (\vt, \Grad \vu) : \Grad \vu } 
\ageq \| \Grad \vu \|^2_{L^2(R^3; R^3)} = \| \vu \|^2_{D^{1,2}_0(R^3; R^3)}.
\end{equation}
Similarly, 
\begin{align} 
- \intR{ \frac{1}{\vt^2} \vc{q} \cdot \Grad \vt } &=  
\intR{ \frac{\kappa(\vt)}{\vt^2} |\Grad \vt|^2 } \ageq 
\| \Grad \vt^{\frac{3}{2}} \|^2_{L^2(R^3, R^3)} + 
\| \Grad \log(\vt) \|^2_{L^2(R^3, R^3)} \br 
&\ageq  
\| \log(\vt) - \log(\vt_\infty) \|^2_{D^{1,2}_0 (R^3)} + \| \vt - \vt_\infty \|^2_{D^{1,2}_0 (R^3) }. 
\label{D2}
\end{align}	

The choice $\tvt = \vt_\infty$ in the ballistic energy inequality \eqref{b12} yields
\begin{align}
	\frac{\D}{\dt}& \intR{ \left[ \frac{1}{2} \vr |\vu|^2 + \vr e(\vr, \vt) - 
		\vt_\infty \vr s(\vr, \vt) + \frac{a}{3} \vt_\infty^4 \right] } \dt \br
	&+ \intR{ \frac{\vt_\infty}{\vt} \left(\mathbb{S}(\vt, \Grad \vu) : \Grad \vu - \frac{\vc{q} \cdot \Grad \vt}{\vt} 
		\right)} + \lambda \intR{ \left(\frac{\vt - \vt_\infty}{\vt} \right) (\vt^4 - \vt_\infty^4)         }\dt \leq 0
	\label{D2a}
\end{align}
in $\mathcal{D}'(0, \infty)$.
In addition, 
separating the radiation components we get 
\begin{align}
	\frac{\D}{\dt}& \intR{ \left[ \left( \frac{1}{2} \vr |\vu|^2 + \vr e_m(\vr, \vt) - 
		\vt_\infty \vr s_m(\vr, \vt) \right) + \left( a \vt^4 - \vt_\infty \frac{4a}{3} \vt^3 + \frac{a}{3} \vt_\infty^4 \right) \right] } \dt \br
	&+ \intR{ \frac{\vt_\infty}{\vt} \left(\mathbb{S}(\vt, \Grad \vu) : \Grad \vu + \frac{\kappa(\vt) |\Grad \vt|^2}{\vt} 
		\right)} + \lambda \intR{ \left(\frac{\vt - \vt_\infty}{\vt} \right) (\vt^4 - \vt_\infty^4)         }\dt \leq 0
	\label{e2}
\end{align}
in $\mathcal{D}'(0, \infty)$. 

It is easy to check that 
\begin{equation} \label{e2a}
	\left(  \vt^4 - \vt_\infty \frac{4}{3} \vt^3 + \frac{1}{3} \vt_\infty^4 \right) \ageq \left\{ \begin{array}{l} (\vt - \vt_\infty)^2 
	\ \mbox{if} \ \vt_\infty/2 \leq \vt \leq 2 \vt_\infty \\ \\
	(1 + \vt^4) \ \mbox{otherwise.}\ \end{array} \right.  
\end{equation}

Unfortunately, the ballistic energy 
\[
E_B \left(\vr, \vt, \vu \Big| \vt_\infty \right)
=
\left( \frac{1}{2} \vr |\vu|^2 + \vr e_m(\vr, \vt) - 
\vt_\infty \vr s_m(\vr, \vt) \right)
\]
may become negative due to the logarithmic part of the entropy $s_m$, cf. Lemma \ref{LS2}. Instead, we consider a suitable cut--off 
\[
E_{B,k} \left(\vr, \vt, \vu \Big| \vt_\infty \right)
=
\left( \frac{1}{2} \vr |\vu|^2 + \vr e_m(\vr, \vt) - 
\vt_\infty \vr s_m(\vr, \vt) - \vt_\infty P'(0) \vr L_k(\vr) \right),
\]
where $L_k$ was introduced in \eqref{b13b}.

\begin{Lemma} \label{Le1}
	Given $k > 0$, there exists $\beta(k)$ such that 
	\begin{align}
	E_{B,k} &\left(\vr, \vt, \vu \Big| \vt_\infty \right) + \beta(k) \vr \br &=
	\left( \frac{1}{2} \vr |\vu|^2 + \vr e_m(\vr, \vt) - 
	\vt_\infty \vr s_m(\vr, \vt) - \vt_\infty P'(0) \vr L_k(\vr) \right) + \beta(k) \vr \geq 0.
	\nonumber
	\end{align}
	
\end{Lemma}

\begin{proof}
	
	First, consider the case 
	\begin{equation} \label{e4}
		\frac{\vr}{\vt^{\frac{3}{2}}} = Z \geq 1.
	\end{equation}
	By virtue of hypothesis \eqref{cc7}, 
	\[
	0 \leq \vr s_m (\vr, \vt) \leq \vr \mathcal{S}(1). 
	\]
	Seeing that 
	\[
	- \vr L_k(\vr) \geq - c(k) \vr 
	\]
	we conclude the desired result follows in the case \eqref{e4}. 
	
	If 
	\begin{equation} \label{e5}
		\frac{\vr}{\vt^{\frac{3}{2}}} = Z < 1, 
	\end{equation}
	we write, in accordance with the decomposition \eqref{S3a}, 
	\begin{align}
		&\left[ \frac{1}{2} \vr |\vu|^2 + \vr e_m(\vr, \vt) - 
		\vt_\infty \vr s_m(\vr, \vt) - \vt_\infty P'(0) \vr L_k(\vr) \right] \br 
		&\quad = \left[ \frac{1}{2} \vr |\vu|^2 + \vr e_m(\vr, \vt) - 
		\vt_\infty \vr \mathcal{S}_B \left( \frac{\vr}{\vt^{\frac{3}{2}}} \right) \right] + \Big[ \vt_\infty P'(0) \vr \log(\vr) - \vt_\infty P'(0) \vr L_k(\vr) \Big] \br 
		&\quad - \frac{3}{2} \vt_\infty P'(0) \vr \log(\vt).
		\nonumber
	\end{align}
	
	Now, in accordance with \eqref{S3b}, 
	\[
	\vt_\infty \vr \left| \mathcal{S}_B \left( \frac{\vr}{\vt^{\frac{3}{2}}} \right) \right| \aleq \vr.
	\]
	Moreover, by virtue of \eqref{S2a}, 
	\[
	\vr e_m (\vr, \vt) \ageq \vr^{\frac{5}{3}} + \vr \vt.
	\]
	Seeing that 
\[
\vt_\infty P'(0) \vr \log(\vr) - \vt_\infty P'(0) \vr L_k(\vr) = 0 \ \mbox{if}\ 
\vr \leq k, 
\]	
and
\[	
\log(\vt) \leq 0 \ \mbox{for}\ \vt < 1,
\]
we conclude 
	\[
	\vr e_m (\vr, \vt) + \beta(k) \vr + 
	\Big[ \vt_\infty P'(0) \vr \log(\vr) - \vt_\infty P'(0) \vr L_k(\vr) \Big] \br 
	- \frac{3}{2} \vt_\infty P'(0) \vr \log(\vt) \geq 0
	\]
	for a suitable $\beta(k)$.	
\end{proof}	 

In view of Lemma \ref{Le1}, inequality \eqref{e2} can be rewritten in the form
\begin{align}
	\frac{\D}{\dt}& \intR{ \left[ E_{B,k} \left(\vr, \vt, \vu \Big| \vt_\infty \right) + \left(a \vt^4 - \vt_\infty \frac{4a}{3} \vt^3 + \frac{a}{3} \vt_\infty^4  \right)  \right] } \br 
		&+ \frac{\D }{\dt} \intR{ \vt_\infty P'(0) \vr L_k(\vr)  }  \br
	&+ \intR{ \frac{\vt_\infty}{\vt} \left(\mathbb{S}(\vt, \Grad \vu) : \Grad \vu - \frac{\vc{q} \cdot \Grad \vt}{\vt} 
		\right)} + \lambda \intR{ \left(\frac{\vt - \vt_\infty}{\vt} \right) (\vt^4 - \vt_\infty^4)         }\dt \leq 0.
	\nonumber
\end{align}

Next, we use the relations \eqref{b13}--\eqref{b16} to deduce
\begin{align}
	\frac{\D}{\dt}& \intR{ \left[ E_{B,k} \left(\vr, \vt, \vu \Big| \vt_\infty \right) + \left(a \vt^4 - \vt_\infty \frac{4a}{3} \vt^3 + \frac{a}{3} \vt_\infty^4  \right) + \beta(k) \vr  \right] } \br 
	&+ \intR{ \frac{\vt_\infty}{\vt} \left(\mathbb{S}(\vt, \Grad \vu) : \Grad \vu - \frac{\vc{q} \cdot \Grad \vt}{\vt} 
		\right)} + \lambda \intR{ \left(\frac{\vt - \vt_\infty}{\vt} \right) (\vt^4 - \vt_\infty^4)         }\dt\br
	&\leq 	\intR{ \vt_\infty P'(0) T_k(\vr) \Div \vu } \leq 
	\vt_\infty P'(0) \| \Div \vu \|_{L^2(R^3)} m_0 \sqrt{k}.
	\label{e6}
\end{align}
	
Finally, integrating \eqref{e6} in time we may infer that  
\[
\limsup_{T \to \infty} \frac{1}{T} \int_1^T \intR{ \frac{1}{\vt} \left[ \left(\mathbb{S}(\vt, \Grad \vu) : \Grad \vu + \frac{\kappa(\vt) |\Grad \vt|^2}{\vt} 
	\right) + (\vt - \vt_\infty)  (\vt^4 - \vt_\infty^4)      \right]   }\dt
\aleq k 	 
\]
for any $k > 0$. As $k > 0$ was arbitrary, we get 
\begin{equation} \label{e7}
\lim_{T \to \infty} \frac{1}{T} \int_1^T \intR{ \frac{1}{\vt} \left[ \left(\mathbb{S}(\vt, \Grad \vu) : \Grad \vu + \frac{\kappa(\vt) |\Grad \vt|^2}{\vt} 
	\right) + (\vt - \vt_\infty)  (\vt^4 - \vt_\infty^4)      \right]   }\dt
= 0.	 
\end{equation}
The convergence stated in \eqref{e7} together with \eqref{D1}, \eqref{D2} yield 
\begin{equation} \label{e8}
	\lim_{T \to \infty} \frac{1}{T} \int_1^T \left[  \| \vt - \vt_\infty \|_{W^{1,2}(R^3)}^2 + \| \vu \|_{D^{1,2}_0 (R^3; R^3)}^2 \right] = 0.
	\end{equation}
	
\begin{Remark} \label{eR1}
A short inspection of the proof of \eqref{e8} reveals that a slightly weaker result, specifically, 
\begin{equation} \label{e9}
	\lim_{T \to \infty} \frac{1}{T} \int_1^T \left[  \| \vt - \vt_\infty \|_{D^{1,2}_0(R^3)}^2 + \| \vu \|_{D^{1,2}_0 (R^3; R^3)}^2 \right] = 0,
\end{equation}
holds in the absence of radiation terms - $a = \lambda = 0$.	
	\end{Remark}	

\subsection{Pressure estimates}
\label{p}

The next goal is to use 
\begin{equation} \label{p1}
\bfphi = \Grad \Del^{-1} [b(\vr)] 
\end{equation}
as a test function in the momentum balance \eqref{b10} for a suitable smooth function $b(\vr)$.
To this end, we replace the pressure $p(\vr, \vt)$ in the weak formulation 
of the momentum balance \eqref{b9} by  
\[
p(\vr, \vt) = p_m (\vr, \vt) + \frac{a}{3} \vt^4 - \frac{a}{3} \vt^4_\infty. 
\]
Writing 
\[
\vt^4 - \vt^4_\infty = \vt^4 - 4 \vt^3_\infty (\vt - \vt_\infty) -\vt^4_\infty
+ 4 \vt^3_\infty (\vt - \vt_\infty)
\]
we observe that the pressure belongs to the space $L^1 + L^2$ as the total energy 
is bounded. This observation justifies the following identity 
\begin{align}
	\int_{\tau_1}^{\tau_2} &\intR{ \left( p_m(\vr, \vt)  + \vr |\vu|^2 
	+ \frac{a}{3} \vt^4 - \frac{a}{3} \vt^4_\infty	 \right) b(\vr) } \dt = - \sum_{i \ne j} \int_{\tau_1}^{\tau_2} \intR{ \vr u_i u_j \partial_{x_i} \Del^{-1} \partial_{x_j} [b(\vr)] } \dt	\br 
	&+ \int_{\tau_1}^{\tau_2} \intR{ \mathbb{S}(\vt ,\Grad \vu) : \Grad \Grad \Del^{-1}[b(\vr)] } \dt - \left[ \intR{ \vr \vu \cdot \Grad \Del^{-1}[b(\vr)] } 
	\right]_{t = \tau_1}^{t = \tau_2} \br 
	&+ \int_{\tau_1}^{\tau_2} \intR{ \vr \vu \cdot \Grad \Del^{-1}[ \partial_t b(\vr) ] } \dt,\ 0 < \tau_1 < \tau_2,
	\label{p2}	
\end{align}
obtained by using $\bfphi$ defined in \eqref{p1} as a test function in \eqref{b10} as long as 
\[
b(\vr) \in L^\infty \cap L^2 (R^3)
\]
and all integrals on the right--hand side of \eqref{p2} converge.

Using H\" older's inequality we get   
\begin{align}
&\left| \int_{\tau_1}^{\tau_2} \intR{ \vr u_i u_j \partial_{x_i} \Del^{-1} \partial_{x_j} [b(\vr)] } \dt \right| \br &\quad \leq \sup_{t \in (\tau_1, \tau_2)}
\| \vr \|_{L^{\frac{5}{3}}(R^3)} \int_{\tau_1}^{\tau_2}\| \vu \|^2_{L^6(R^3; R^3)}
\dt 
\sup_{t \in (\tau_1, \tau_2)} \Big\| \Grad \Del^{-1} \Grad [b(\vr)] \Big\|_{L^{15}(R^3)}.
\label{p3}
\end{align}
Similarly, 
\begin{align}
&\left|	\int_{\tau_1}^{\tau_2} \intR{ \mathbb{S}(\vt, \Grad \vu) : \Grad \Grad \Del^{-1}[b(\vr)] } \dt \right| \br
&\quad \aleq \int_{\tau_1}^{\tau_2} \intR{ \Big( |\vt - \vt_\infty| |\Grad \vu| + 
	|\Grad \vu| \Big) |\Grad \Grad \Del^{-1}[b(\vr)]| } \dt 
\br
&\quad \aleq \int_{\tau_2}^{\tau_2} \| \vt - \vt_\infty\|_{L^6(R^3)} 
\| \Grad \vu \|_{L^2(R^3; R^{3 \times 3})}  \dt \sup_{t \in (\tau_1, \tau_2)} \Big\| \Grad \Del^{-1} \Grad [b(\vr)] \Big\|_{L^{3}(R^3)}
\br
&\quad +  \int_{\tau_1}^{\tau_2} \| \Grad \vu \|_{L^2(R^3; R^{3 \times 3})}  \dt \sup_{t \in (\tau_1, \tau_2)} \Big\| \Grad \Del^{-1} \Grad [b(\vr)] \Big\|_{L^{2}(R^3)},
\label{p4}	
	\end{align}
and	
\begin{align} \label{p5}
\left|	\left[ \intR{ \vr \vu \cdot \Grad \Del^{-1}[b(\vr)] } 
	\right]_{t = \tau_1}^{t = \tau_2} \right| \leq
\sup_{t \in (\tau_1, \tau_2)} \intR{ \vr |\vu| } 
 \sup_{t \in (\tau_1, \tau_2)} \Big\| \Grad \Del^{-1} [b(\vr)] \Big\|_{L^{\infty}(R^3)}.
	\end{align}	
	
As the total mass $m_0$ is finite, meaning $\| \vr(t, \cdot) \|_{L^1(R^3)} \aleq 1$ uniformly for $t > 0$, we can choose 
\begin{equation} \label{p5c}
b(\vr) = \min \left\{ \vr , \vr^{\frac{1}{\alpha}} \right\} ,\ \alpha > 0. 
\end{equation}
The standard elliptic estimates yield 	
\begin{align} 
\| \Grad \Del^{-1} \Grad b(\vr) (t, \cdot) \|_{L^q(R^3)} \aleq 1 
\ \mbox{for all}\ t  > 0,\ \mbox{and any}\  1 < q < \alpha , \label{p5a} \\
\| \Grad \Del^{-1} b(\vr) (t, \cdot) \|_{L^r (R^3, R^3)} \aleq 1
\ \mbox{for all}\ t > 0, \ \mbox{and any}\ \frac{3}{2} < r \leq \infty 
\ \mbox{as soon as}\ \alpha > 3.
\label{p5b}
\end{align}
		
Thus if $b(\vr)$ is given by \eqref{p5a} 
we conclude	
\begin{align} 
	&\int_{\tau_1}^{\tau_2} \intR{ \left( p_m(\vr, \vt)  + \vr |\vu|^2 
	+ \frac{a}{3} \vt_4 - \frac{4a}{3} \vt_\infty^3 (\vt - \vt_\infty) - \frac{a}{3} \vt^4_\infty	 \right) b(\vr) } \dt \br
	& \leq c(m_0) \left( \sup_{t \in (\tau_1, \tau_2)}
	\| \vr \|_{L^{\frac{5}{3}}(R^3)} \int_{\tau_1}^{\tau_2}\| \vu \|^2_{L^6(R^3; R^3)}   \dt + \sup_{t \in (\tau_1, \tau_2)} \intR{ \vr |\vu| } \right)	\br
&+ c(m_0) \left(  \int_{\tau_1}^{\tau_2} \| \Grad \vu \|_{L^2(R^3; R^{3 \times 3})} \dt +     \int_{\tau_1}^{\tau_2} \| \Grad \vu \|^2_{L^2(R^3; R^{3 \times 3})} \dt +  \int_{\tau_1}^{\tau_2} \| \vt - \vt_\infty \|^2_{L^6(R^3)} \dt                                      \right) \br	
&+ \left| \int_{\tau_1}^{\tau_2} \intR{ \vr \vu \cdot \Grad \Del^{-1}[ \partial_t b(\vr) ] } \dt \right|.
	\label{p6}
\end{align}

Finally, we use the renormalized equation of continuity to handle the last term 
in \eqref{p6}:
\begin{align}
\int_{\tau_2}^{\tau_2} &\intR{ \vr \vu \cdot \Grad \Del^{-1}[ \partial_t b(\vr) ] } \dt = - \int_{\tau_1}^{\tau_2} \intR{ \vr \vu \cdot \Grad \Del^{-1} \Div[  b(\vr) \vu ]  } \dt \br
&+ \int_{\tau_1}^{\tau_2} \intR{ \vr \vu \cdot \Grad \Del^{-1} \left[ \left( b(\vr) - b'(\vr) \vr \right) \Div \vu  \right]  } \dt. 
\nonumber
\end{align}
Since $b'(\vr) \vr \approx b(\vr)$ (except $\vr = 1$), 
both integrals on the right--hand side can be estimated by means of H\" older's inequality and the standard Sobolev embedding theorem in the same way as in \eqref{p4}. Thus 
we may infer that 
\begin{align} 
	&\int_{\tau_1}^{\tau_2} \intR{ \left( p_m(\vr, \vt)  + \vr |\vu|^2 
		+ \frac{a}{3} \vt_4  - \frac{4a}{3} \vt_\infty^3 (\vt - \vt_\infty) - \frac{a}{3} \vt^4_\infty	 \right) b(\vr) } \dt \br
	& \leq c(m_0) \left( \sup_{t \in (\tau_1, \tau_2)}
	\| \vr \|_{L^{\frac{5}{3}}(R^3)} \int_{\tau_1}^{\tau_2}\| \vu \|^2_{L^6(R^3; R^3)}   \dt + \sup_{t \in (\tau_1, \tau_2)} \| \sqrt{\vr} \vu \|_{L^2(R^3; R^3)} \right)	\br
	&+ c(m_0) \left(  \int_{\tau_1}^{\tau_2} \| \Grad \vu \|_{L^2(R^3; R^{3 \times 3})} \dt +     \int_{\tau_1}^{\tau_2} \| \Grad \vu \|^2_{L^2(R^3; R^{3 \times 3})} \dt +  \int_{\tau_1}^{\tau_2} \| \vt - \vt_\infty \|^2_{L^6(R^3)} \dt  \right). 
	\label{p7}
\end{align}
where we have used 
\[
\intR{ \vr |\vu| } \leq \| \sqrt{\vr} \|_{L^2(R^3)} 
 \| \sqrt{\vr} \vu \|_{L^2(R^3; R^3)} = \sqrt{m_0}  \| \sqrt{\vr} \vu \|_{L^2(R^3; R^3)}.
\]	

\subsection{Auxiliary inequalities}

We recall the inequality 
\begin{align}
	&\int_{\tau_1}^{\tau_2} \intR{ \vr |\log (\vt) | } \dt 
	\leq \int_{\tau_1}^{\tau_2} \intR{ \vr |\log (\vt) - \log(\vt_\infty) | } \dt
	+ \int_{\tau_1}^{\tau_2} \intR{ \vr |\log (\vt_\infty) | } \dt \br
&\quad \leq \sup_{t \in (\tau_1, \tau_2)} \| \vr \|_{L^{\frac{6}{5}}(R^3)} 
\int_{\tau_1}^{\tau_2} \| \log (\vt) - \log(\vt_\infty) \|_{L^6(R^3)} \dt
	+ \int_{\tau_1}^{\tau_2} \intR{ \vr |\log (\vt_\infty) | } \dt.
	\label{p8}
\end{align}
Similarly, 
\begin{align}
	&\int_{\tau_1}^{\tau_2} \intR{ \vr \vt } \dt 
	\leq \int_{\tau_1}^{\tau_2} \intR{ \vr |\vt - \vt_\infty | } \dt
	+ \int_{\tau_1}^{\tau_2} \intR{ \vr \vt_\infty } \dt \br
	&\quad \leq \sup_{t \in (\tau_1, \tau_2)} \| \vr \|_{L^{\frac{6}{5}}(R^3)} 
	\int_{\tau_1}^{\tau_2} \| \vt - \vt_\infty \|_{L^6(R^3)} \dt
	+ \int_{\tau_1}^{\tau_2} \intR{ \vr \vt_\infty  } \dt, 
	\label{p9}
\end{align}
and 
\begin{equation}
	\int_{\tau_1}^{\tau_2} \intR{ \vr |\vu|^2 } \dt 
	\quad \leq \sup_{t \in (\tau_1, \tau_2)} \| \vr \|_{L^{\frac{3}{2}}(R^3)} 
	\int_{\tau_1}^{\tau_2} \| \vu \|_{L^6(R^3)}^2 . 
	\label{p10}
\end{equation}

\section{Bounded absorbing set}
\label{P}

Going back to \eqref{e6}, we introduce the truncated energy 
\begin{align} 
F_k (\vr, \vt, \vu) &= 
E_{B,k} \left( \vr, \vt, \vu \Big| \vt_\infty \right) + \beta(k) \vr 
+  a \vt^4 - \vt_\infty \frac{4a}{3} \vt^3 + \frac{a}{3} \vt_\infty^4 \br &=
\frac{1}{2} \vr |\vu|^2 + \vr e_m(\vr, \vt) - 
\vt_\infty \vr s_m(\vr, \vt) - \vt_\infty P'(0) \vr L_k(\vr) + \beta(k) \vr \br 
&\quad + a \vt^4 - \vt_\infty \frac{4a}{3} \vt^3 + \frac{a}{3} \vt_\infty^4 .
\label{P1} 
\end{align}
As stated in Lemma \ref{Le1}, the function $F_k (\vr, \vt, \vu) \geq 0$. Our goal is to show that 
\begin{equation} \label{P1a}
	\limsup_{t \to \infty} \intR{ F_k (\vr, \vt, \vu) } \leq F_{\infty}(m_0, k)
\end{equation}
for any global in time weak solution of the NSF system. 

We start with an auxiliary estimate. 
\begin{Lemma} \label{LP2}
There is $\omega(k) > 0$ such that 
\[
\omega(k) \left( \vr^{\frac{5}{3}} + \vr + \vr \vt + \vr |\log (\vt) | \right) 
\geq E_{B,k} 
\]	
	\end{Lemma}
	
\begin{proof}
It is enough to show that 
\[
\omega(k) \left( \vr^{\frac{5}{3}} + \vr + \vr \vt + \vr |\log (\vt) | \right) 
\geq 2 \vt_\infty \left| \vr s_m(\vr, \vt) + P'(0) \vr L_k(\vr) \right| 
\]	
for a suitable $\omega(k) > 0$.

First suppose that 
\[
\frac{\vr}{\vt^{\frac{3}{2}}} = Z \leq 1. 
\]	
By virtue of Lemma \ref{LS2}, 
\[
\vr s_m (\vr, \vt) + P'(0) \vr L_k(\vr)= \vr \mathcal{S}_B \left( \frac{\vr}{\vt^{\frac{3}{2}}} \right) + \frac{3}{2} P'(0) \vr \log(\vt) + P'(0) \vr L_k(\vr)
-P'(0) \vr \log(\vr). 
\]
As $\mathcal{S}_B$ is bounded and 
\[
\Big|  P'(0) \vr L_k(\vr)
-P'(0) \vr \log(\vr) \Big| \leq c(k) \left( \vr + \vr^{\frac{5}{3}} \right),
\]
the desired result follows. 

If 
\begin{equation} \label{P2a}
\frac{\vr}{\vt^{\frac{3}{2}}} = Z \geq 1, 
\end{equation}
we have 
\[
0 \leq s_m (\vr, \vt) \leq \mathcal{S}(1). 
\]
Finally, 
\[
\vr |L_k(\vr)| \leq - \vr \log(\vr) \ \mbox{if}\ \vr \leq 1 , 
\vr |L_k(\vr)| \aleq \vr \ \mbox{if}\ \vr > 1, 
\]
where, in accordance with \eqref{P2a}, 
\[
- \log(\vr) \leq - \frac{3}{2} \log(\vt) \leq \frac{3}{2} |\log(\vt)|, 
\]
which completes the proof.

	\end{proof}	
	
The following result is crucial for the existence of a bounded absorbing set.	
\begin{Lemma} \label{LP1}	
Let $k > 0$ be given. Suppose
	\begin{equation} \label{P2}
		\intR{ F_k(\vr, \vt, \vu) (\tau, \cdot) }
		- \intR{ F_k(\vr, \vt, \vu) (\tau +1 , \cdot) }
		\leq K.
	\end{equation}	

	Then there exists a constant $L(K,k, m_0)$ such that 
	\begin{equation} \label{P3}
		\sup_{t \in (\tau, \tau + 1)} \intR{ F_k (\vr, \vt, \vu ) (t, \cdot) } 
		\leq L (K,k, m_0).
	\end{equation}

\end{Lemma}

\begin{proof}
	
First observe that hypothesis \eqref{P2} together with the modified energy inequality \eqref{e6} yield 
\begin{equation} \label{P4}
\int_{\tau}^{\tau + 1} \intR{ \left( |\Grad \vu|^2 + \frac{\kappa(\vt) |\Grad \vt|^2}{\vt^2} 
	\right)}\dt + \int_{\tau}^{\tau + 1} \intR{ \left(\frac{\vt - \vt_\infty}{\vt} \right) (\vt^4 - \vt_\infty^4)         }\dt 	\leq c(m_0, k, K).
\end{equation}		
In particular, applying Poincar\' e and Sobolev inequality, we deduce 
\begin{align} \label{P5}
	\int_{\tau}^{\tau + 1} &\left( \| \vu \|^2_{L^6(R^3;R^3)} + \| \Grad \vu \|^2_{L^2(R^3; R^{3 \times 3})} + \| \vt - \vt_\infty \|^2_{L^6(R^3)} + \| \log(\vt) - \log(\vt_\infty) \|^2_{L^6(R^3)}
		\right)\dt \\ 	&\leq c(m_0, k, K) \nonumber.
\end{align}	

In the present setting, the pressure estimates stated in \eqref{p7} imply 
\begin{align} 
	&\int_{\tau}^{\tau + 1} \intR{ \vr^{\frac{5}{3}} b(\vr) } \dt \br
	& \leq c(m_0,k,K) \left( \sup_{t \in (\tau, \tau + 1)}
	\| \vr \|_{L^{\frac{5}{3}}(R^3)}  + \sup_{t \in (\tau, \tau + 1)} \| \sqrt{\vr} \vu \|_{L^2(R^3; R^3)} + 1 \right), 
	\label{P6}
\end{align}
where $b(\vr) = \left\{ \vr, \vr^{\frac{1}{\alpha}} \right\},\ \alpha > 3.$
As $\vr^{\frac{5}{3}} b(\vr) \leq \vr$ for $\vr \leq 1$, we easily deduce from 
\eqref{P6}: 
\begin{align} 
	&\int_{\tau}^{\tau + 1} \intR{ \vr^{\frac{5}{3}}  } \dt \br
	& \leq c(m_0,k,K) \left( \sup_{t \in (\tau, \tau + 1)}
	\| \vr \|_{L^{\frac{5}{3}}(R^3)}  + \sup_{t \in (\tau, \tau + 1)} \| \sqrt{\vr} \vu \|_{L^2(R^3; R^3)} + 1 \right). 
	\label{P6a}
\end{align}

Using the bounds \eqref{P4}, \eqref{P5}, together with Lemma \ref{LP2}, 
we may infer that  
\begin{align} 
	&\int_{\tau}^{\tau + 1} \intR{ F_k(\vr, \vt, \vu)  } \dt \br
	& \leq c(m_0,k,K) \left( \left( \sup_{t \in (\tau, \tau + 1)} \intR{ F_k(\vr, \vt, \vu)  } \right)^{\frac{3}{5}} + \left( \sup_{t \in (\tau, \tau + 1)} \intR{ F_k(\vr, \vt, \vu)  } \right)^{\frac{1}{2}}
	 + 1 \right).
	\label{P7}
\end{align}

Finally we claim that hypothesis \eqref{P2} implies 
\begin{equation} \label{P10}
	\sup_{t \in (\tau, \tau + 1)} 	\intR{ F_k(\vr, \vt, \vu) (t, \cdot) } \leq 
	c(m_0,k,K) \left(1 + \int_{\tau}^{\tau + 1} \intR{ F_k (\vr, \vt, \vu) } \dt \right).  
\end{equation}
Indeed, by the mean value theorem, there is $\xi \in [\tau, \tau + 1]$ such that 
\[
\int_{\tau}^{\tau + 1} \intR{ F_k(\vr, \vt, \vu) }
= \intR{ F_k(\vr, \vt, \vu) (\xi, \cdot) }.
\]	
Now, it follows from \eqref{e6} that 
\begin{equation} \label{P11}
\intR{ F_k(\vr, \vt, \vu) (t, \cdot) } 
\leq \int_{\tau}^{\tau + 1} \intR{ F_k(\vr, \vt, \vu) } + c(m_0,k)
\end{equation}
for all $\xi \leq t \leq \tau + 1$. 
However, by virtue of hypothesis \eqref{P2}, 
\begin{align}
\intR{ F_k(\vr, \vt, \vu) (\tau, \cdot) }
&\leq K + \intR{ F_k(\vr, \vt, \vu) (\tau+1, \cdot) } \br 
&\leq 
c(K,k, m_0) + \int_{\tau}^{\tau + 1} \intR{ F_k(\vr, \vt, \vu) },
\nonumber
\end{align}
and the proof of \eqref{P10} follows by the arguments used for showing \eqref{P11}.

Relations \eqref{P7}, \eqref{P10} yield the desired bound \eqref{P3}.

\end{proof}

Fix $K = 1$ in \eqref{P2}. As $F_k(\vr, \vt, \vu) \geq 0$, it follows there is a positive time $T$ such that 
\[
\sup_{t \in (T, T + 1)} \intR{ F_k (\vr, \vt, \vu ) (t, \cdot) } 
\leq L (k, m_0). 
\]
Now, it follows from Lemma \ref{LP1} that 
\[
\intR{ F_k (\vr, \vt, \vu ) (T+1, \cdot) } \leq  L (k, m_0), 
\]
and, by recursion,
\[
\intR{ F_k (\vr, \vt, \vu ) (T+n, \cdot) } \leq  L (k, m_0),\ n = 1,2,\dots.
\]
Consequently, applying the energy inequality \eqref{e6} on each interval 
$(T+n, T + n + 1)$, we deduce the existence of a universal constant $F_\infty(m_0, k)$ claimed in \eqref{P1a}. 

We have shown the following result. 

\begin{Proposition} [{\bf Bounded absorbing set}] \label{PP1}
Under the hypotheses of Theorem \ref{MT1}, there exists a constant $F_\infty(k, m_0)$ such that 
\begin{equation} \label{P20}
	\limsup_{t \to \infty} \intR{ F_k(\vr, \vt, \vu) } \leq F_\infty(k, m_0), 
\end{equation}		
where $F_k$ is the modified energy functional introduced in \eqref{P1}.	
	\end{Proposition}

\section{Long time behaviour, conclusion}
\label{L}

In view of \eqref{e9}, the proof of Theorem \ref{MT1} reduces to showing 
\begin{equation} \label{L1}
\frac{1}{T} \int_1^T \| \vr \|^{\frac{5}{3}}_{L^{\frac{5}{3}}(R)} \dt \to 0 
\ \mbox{as}\ T \to \infty.	
\end{equation}
Going back to the pressure estimates \eqref{p7} and using the uniform bounds 
\eqref{P10} we deduce
\[ 
\frac{1}{T} \int_1^T \intR{ \vr^{\frac{5}{3}} b(\vr) } \dt \to 0 
\ \mbox{as}\ T \to \infty.	
\]
As 
\[
\vr^{\frac{5}{3}} \leq \delta \vr + c(\delta) \vt^{\frac{5}{3}} b(\vr)
\]
for any $\delta > 0$, the convergence claimed in \eqref{L1} follows.

We have shown Theorem \ref{MT1}. 

\section{Concluding remarks}
\label{C} 

The result stated in Theorem \ref{MT1} is ``optimal'' in the sense that  
the ballistic energy appearing in \eqref{D2}, or, more specifically, the quantity 
\[
\intR{ \vr \log (\vr) (t, \cdot) }, 
\]
cannot be bounded below for $t \to \infty$. 
It follows from Theorem \ref{MT1} that there is a sequence of times 
$T_n \to \infty$ such that 
\begin{equation} \label{C1}
\vr (T_n, \cdot) \to 0 \ \mbox{in}\ L^{\frac{5}{3}}(R^3).
\end{equation}
We claim that 
\[
\intR{ \vr \log (\vr) (T_n, \cdot) } \to - \infty. 
\]
Indeed, as the total mass is a constant of motion, 
\[
m_0 = \intR{ \vr(T_n, \cdot) } = 
\intR{ \mathds{1}_{\vr \leq \delta } \vr (T_n, \cdot) } + \intR{\mathds{1}_{\vr > \delta } \vr (T_n, \cdot)}, 
\] 
where, in view of \eqref{C1}, 
\[
\intR{\mathds{1}_{\vr > \delta } \vr (T_n, \cdot)} \to 0 
\] 
for any $\delta > 0$. Finally, 
\[
\intR{ \mathds{1}_{\vr \leq \delta } \vr (T_n, \cdot) } \leq 
\frac{1}{|\log (\delta)|} \intR{ \vr |\log(\vr) | (T_n, \cdot) },\ 
|\log (\delta)| \to \infty \ \mbox{as}\ \delta \to 0.
\]

\end{document}